\documentclass{amsart} 

\usepackage{amsmath,amssymb,amsthm} 
\usepackage{graphicx} 
\usepackage{subfig}
\usepackage[arrow, matrix, curve]{xy} 
\usepackage{xcolor} 
\usepackage{makecell} 
\usepackage{multirow} 
\usepackage{enumerate}
\usepackage{booktabs} 
\usepackage{siunitx} 
\usepackage{hyperref} 

\DeclareMathOperator{\lk}{lk}

\newcommand{\Z}{\mathbb{Z}}
\newcommand{\Q}{\mathbb{Q}}
\newcommand{\N}{\mathbb{N}}

\makeatletter
\newcommand{\Vast}{\bBigg@{2.5}} 
\makeatother

\makeatletter
\newcommand{\smallbullet}{} 
\DeclareRobustCommand\smallbullet{%
  \mathord{\mathpalette\smallbullet@{0.7}}
}
\newcommand{\smallbullet@}[2]{%
  \vcenter{\hbox{\scalebox{#2}{$\m@th#1\bullet$}}}%
}
\makeatother

\newtheoremstyle{thm}{}{}{\itshape}{}{\bfseries}{}{ }{} 
\newtheoremstyle{definition}{}{}{}{}{\bfseries}{}{ }{} 

\theoremstyle{thm}
\newtheorem{Theorem}{Theorem}[section]
\newtheorem{thm}[Theorem]{Theorem}
\newtheorem{lem}[Theorem]{Lemma}
\newtheorem{prop}[Theorem]{Proposition}
\newtheorem{cor}[Theorem]{Corollary}
\newtheorem*{Theorem-ohne}{Theorem}

\newtheorem{ques}[Theorem]{Question}

\newtheorem*{lemma:hyp_reduction}{Lemma~\ref{lemma:hyp_reduction}}

\theoremstyle{definition}
\newtheorem{defi}[Theorem]{Definition}
\newtheorem{rem}[Theorem]{Remark}
\newtheorem{ex}[Theorem]{Example}

\begingroup\expandafter\expandafter\expandafter\endgroup
\expandafter\ifx\csname pdfsuppresswarningpagegroup\endcsname\relax
\else
\fi

\definecolor{amaranth}{rgb}{0.9, 0.17, 0.31} 
\definecolor{carrotorange}{rgb}{0.93, 0.57, 0.13} 
\definecolor{citrine}{rgb}{0.89, 0.82, 0.04} 
\definecolor{dartmouthgreen}{rgb}{0.05, 0.5, 0.06} 
\definecolor{ballblue}{rgb}{0.13, 0.67, 0.8} 
\definecolor{ceruleanblue}{rgb}{0.16, 0.32, 0.75} 
\definecolor{amethyst}{rgb}{0.6, 0.4, 0.8} 
\definecolor{amber}{rgb}{1.0, 0.75, 0.0} 
\definecolor{burlywood}{rgb}{0.87, 0.72, 0.53} 

\numberwithin{equation}{section}

\begin{document}

\title{Unique Surgery Descriptions along Knots} 

\author{Marc Kegel}
\address{Humboldt Universit\"at zu Berlin, Rudower Chaussee 25, 12489 Berlin, Germany}
\email{kegelmarc87@gmail.com}

\author{Misha Schmalian}
\address{University of Oxford, Andrew Wiles Building, OX2 6GG, UK}
\email{schmalian@maths.ox.ac.uk}



\begin{abstract}
We prove that for any non-trivial knot $K$, infinitely many $r$-surger\-ies $K(r)$ along $K$ have a unique surgery description along a knot. Moreover, we show that for any hyperbolic L-space knot $K$, and infinitely many integer slopes $n$, the manifold $K(n)$ has a unique surgery description. Here we say a 3-manifold M has a {\it unique surgery description along a knot} in $S^3$ if there is a unique pair $(K,r)$ of a knot $K$ and a slope $r$ such that $M$ is orientation-preservingly diffeomorphic to $K(r)$. This generalises the notion of characterising slopes. Conversely, we provide new families of manifolds with several distinct surgery descriptions along knots. More precisely, we construct for every non-zero integer $m$ a knot $K_m$ such that for any integer $n$, the manifold $K_m(m+1/n)$ can also be obtained by surgery on another knot.
\end{abstract}


\makeatletter
\@namedef{subjclassname@2020}{%
  \textup{2020} Mathematics Subject Classification}
\makeatother



\maketitle

\section{Introduction}
Given a link $L$ in $S^3$ a key concept in low-dimensional topology is the process of Dehn surgery along $L$. 
Lickorish and Wallace showed that any oriented closed 3-manifold may be obtained by Dehn surgery along a link \cite{Lickorish, Wallace}. However, it was quickly observed that such a surgery description, even when restricting to knots instead of links, does not have to be unique \cite{Lickorish_sharing_surgery}. In this article, we discuss when surgery descriptions are unique.\\

We say an oriented 3-manifold $M$ has a {\it unique surgery description along a knot} if there exists a unique pair of the isotopy class of a knot $K$ in $S^3$ and a slope $r\in \mathbb Q\cup\{\infty\}$ such that $M$ is orientation-preservingly diffeomorphic to the $r$-surgery $K(r)$ along $K$. To the authors' knowledge, the following is a list of the manifolds known to have unique surgery descriptions. Gabai proved Property R, which implies that $S^1\times S^2$ has a unique surgery description as the $0$-surgery along the unknot \cite{Gabai_Property_R}. Ghiggini proved that the Poincar\'e homology sphere $\Sigma(2,3,5)$ has a unique surgery description as the $1$-surgery along the right-handed trefoil \cite{Ghighini_Poincare_Sphere}. Finally, Baldwin and Sivek showed that for infinitely many knots $K$ in $S^3$ the $0$-surgery $K(0)$ has a unique surgery description \cite{Baldwin_Sivek_0_Surgery}. In particular, the Poincar\'e homology sphere was the only rational homology sphere known to have a unique surgery description along a knot. Our main result shows that there are infinitely many such manifolds as follows.

\begin{thm} \label{thm:main_theorem}
 For any non-trivial knot $K$ in $S^3$ there are infinitely many rational slopes $r\in \mathbb Q\cup\{\infty\}$ such that $K(r)$ has a unique surgery description. 
\end{thm}

The proof of this result is fairly accessible, relying only on elementary algebraic topology, quadratic reciprocity, hyperbolic geometry, and the JSJ-decomposition of knots. In particular, the proof in the case of hyperbolic knots in Section \ref{section:hyperbolic_knots} can be read independently from the remaining paper and contains the novel ideas of the proof. For the remaining article, we will use the following terminology.

\begin{defi} For a knot $K$ in $S^3$, a slope $r\in \mathbb Q\cup\{\infty\}$ is called {\it strongly-characterising} if, whenever $K(r)$ is orientation-preservingly diffeomorphic to the $r'$-surgery $K'(r')$ along some knot $K'$, then $K'$ is isotopic to $K$.
\end{defi}

It has been shown that if $K(r)$ and $K(r')$ are orientation-preservingly diffeomorphic for a non-trivial knot $K$ and $r\neq r'$, then $r=-r'=\pm 2$ \cite{Cosmetic_Daemi_Lidman_Eismeier}. Hence, Theorem \ref{thm:main_theorem} is equivalent to every non-trivial knot $K$ having infinitely many strongly-characterising slopes. \\ 

In comparison, we say that a slope $r\in \mathbb Q\cup \{\infty\}$ is {\it characterising} for a knot $K$ if whenever $K(r)$ is orientation-preservingly diffeomorphic to the $r$-surgery $K'(r)$, then $K'$ is isotopic to $K$. As the name suggests, a strongly-characterising slope is in particular characterising. 
Especially in recent years, there has been extensive work on characterising slopes.
  Every non-trivial slope is known to be characterising for the unknot, the trefoil knots, and the figure-eight knot \cite{KMOS_unknot_char}, \cite{trefoil_eight_char}. Every knot is known to have infinitely many characterising slopes \cite{Lackenby_char} and, in fact, for any knot $K$ the slope $p/q$ is known to be characterising for sufficiently large $|q|$ \cite{Lackenby_char, McCoy_char_torus, Sorya_satellite_char}. On the other hand, Baker and Motegi construct knots for which all integer slopes are non-characterising \cite{Baker_Motegi_non_char}. Moreover, it is known that every non-trivial slope is non-characterising for some knot \cite{Wakelin_Picirillo_Hayden_any_slope}. Additional work on characterising and non-characterising slopes appears for example in \cite{Lickorish_sharing_surgery,Akbulut,Brakes_non_char,Livingston_surgery,Motegi_surgery,Gordon_Luecke,Osoinach,Teragaito,AJLO2,Ni_Zhang_Torus_Char,AJLO,Yasui,MillerPiccirillo,Piccirillo,McCoy_char_hyp,Abe_Tagami,Manolescu_Piccirillo,Wakelin_char,Wakelin_Sorya_effective_char,Baldwin_Sivek_char_5_2,Tagami_RGB,Stein,Baldwin_Sivek_traces,KP,McCoy_char_Floer,traces}.\\

In Section \ref{sec:infinite}, we discuss explicit examples of non-strongly-characterising slopes that are characterising. For example, by considering cables of knots, we see that there are infinitely many slopes that are not strongly-characterising for any knot $K$~\cite{Motegi_surgery}. For an exact statement, see Proposition \ref{thm:cable_slopes}. Moreover, we construct the following family of examples. 

\begin{thm}
    For every $m\in \mathbb Z\setminus\{0\}$ there exists a knot $K_m$ in $S^3$ such that for each $n\in \mathbb Z\setminus\{0\}$ the slope $m+\frac 1n$ is not strongly-characterising for $K_m$. 
\end{thm}

The above theorem is a weaker version of Theorem \ref{thm:all_slopes_precise} and should be seen in contrast to the aforementioned result that for any knot $K$ the slopes $p/q$ with $|q|$ sufficiently large are characterising. \\

Finally, exceptional behavior is observed more often for integral Dehn surgery. In our setup this is, for example, exemplified by all integer slopes being non-characterising in the examples of Baker and Motegi \cite{Baker_Motegi_non_char}. Moreover, Sorya shows that non-integer slopes are characterising for composite knots \cite{Sorya_satellite_char}. In fact, the strongly-characterising slopes $p/q$ we find in Theorem \ref{thm:main_theorem} will have large $|q|$. On the other hand, McCoy shows that for a hyperbolic L-space knot, all sufficiently large integer slopes are characterising \cite{McCoy_char_hyp}. We prove a similar statement for strongly-characterising slopes.

\begin{thm} \label{thm:integer_L_space} For any hyperbolic L-space knot, infinitely many integer slopes are strongly-characterising. 
\end{thm}

We conclude the introduction with some natural questions arising from this article. First, we would like to have a complete classification of strongly-characterising slopes for simple knot types. Note that the resolution of the Berge realization problem~\cite{Greene,Berge} yields such a classification for the unknot. Indeed, a slope $(-p/q)$ is strongly-characterising if and only if the lens space $L(p,q)$ cannot be obtained by surgery on a non-trivial Berge knot. 

\begin{ques}\label{ques:classification}
    What is the classification of strongly-characterising slopes for the trefoil knots and the figure-eight knot?
\end{ques}

Via the cabling construction, one can observe that every knot has infinitely many non-strongly-characterising slopes, see Proposition \ref{thm:cable_slopes}. So a simpler version of the classification question would be to ask if there exist a knot $K$ for which all non-strongly-characterising slopes come from a cabling construction. Note that by Theorem~\ref{thm:untwisting_precise} and Corollary~\ref{cor:unknotting_precise} such a knot $K$ necessarily needs to have unknotting and untwisting number greater than one.

\begin{ques}
    Are there knots $K$ such that, whenever the $r$-surgery $K(r)$ is orienta\-tion-preservingly diffeomorphic to $K'(r')$, then $K$ and $K'$ are isotopic or $K'$ is a cable of $K$?
\end{ques}

Next, we would like to better understand which integer slopes are strongly-characterising.

\begin{ques}
    Do there exist non-trivial knots for which all integer slopes are strongly-characterising?
\end{ques}

\begin{ques}
    Do there exist knots for which every integer slope is characterizing but not strongly-characterising?
\end{ques}

Finally, we would like to raise the algorithmic question of deciding whether a given manifold has a unique surgery description along a single knot. 

\begin{ques}
    Is it algorithmically decidable whether a given slope $r$ of a given knot $K$ is (strongly-)characterising?
\end{ques}

\subsection*{Conventions}
All $3$-manifolds and knots are assumed to be smooth. We consider knots up to isotopy and $3$-manifolds up to orientation-preserving diffeomorphism. For $p/q\in \mathbb Q\cup \{\infty\}$ we denote by $K(p/q)$ the Dehn filling of $S^3\setminus N(K)$ along the slope $p\cdot \mu+q\cdot \lambda$ where $\mu, \lambda\in H_1(\partial N(K))$ are the meridian and Seifert longitude oriented by the standard right-hand convention. When specifying a slope $p/q\in \mathbb Q\cup \{\infty\}$, we always take $p\geq 0$. We will mostly consider unoriented knots and links. However, at certain points (for example to define linking numbers), we will fix orientations for links. It will be straightfoward to see that these choices of orientations will not affect the main results. By the Gordon-Luecke knot complement theorem~\cite{Gordon_Luecke}, the isotopy class of a knot $K$ in $S^3$ is uniquely determined by its exterior $S^3\setminus N(K)$ or, equivalently, its complement $S^3\setminus K$~\cite{Edwards}. We implicitly use this deep theorem in several places throughout the article. 

\subsection*{Acknowledgments}
Both authors would like to thank Chun-Sheng Hsueh, Marc Lackenby, Naageswaran Manikandan, and Steven Sivek for useful discussions and insightful input. Moreover, both authors would like to thank Steven Sivek for communicating Proposition \ref{Sivek_Lemma}.\\
This project began at the \textit{Berlin-Brandenburg Workshop V: Knot Theory and its Application} when the second author visited the Humboldt University Berlin. Both the workshop and the visit were funded by the \textit{Berlin Mathematics Research Center BMS/MATH+} and the \textit{Kolleg Mathematik Physik Berlin (KMPB)}.
The first author is funded by the DFG, German Research Foundation (Project: 561898308).

\section{Constructions of non-strongly-characterising slopes} \label{sec:infinite}

In this section, we will highlight some methods for constructing non-strongly-characterising slopes.  

\subsection{Cable slopes}
First, we observe that there exist infinitely many slopes that are non-strongly-characterising for any knot. This follows from Gordon's results on surgeries on cables~\cite{Gordon_satellite} and was already discussed by Motegi in~\cite{Motegi_surgery}. 

\begin{defi}
The satellite knot with companion $K$ and pattern the torus knot $T_{r,s}\subset S^1\times D^2$ is called the $(r,s)$-{\it cable of $K$} and is denoted by $C_{r,s}(K)$. Here we chose the convention that the pattern $T_{r,s}\subset S^1\times D^2$ is isotopic to the boundary slope $r\cdot(\{1\}\times \partial D^2)+s\cdot (S^1\times\{1\})$. 
\end{defi}    

Note that $C_{r,s}(K)$ is a knot if and only if $r,s$ are coprime. Moreover, if $s=\pm 1$ then $C_{r,s}(K)=K$. Finally $C_{r,s}(K)=C_{-r,-s}(K)$. Hence, we adopt the convention that $s>1$ and that $r,s$ are coprime. For certain slopes, Dehn fillings of a knot $K$ and its cable $C_{r,s}(K)$ can agree. Concretely, we have the following lemma.

\begin{lem} [Corollary 7.3 in~\cite{Gordon_satellite}]
\label{lemma:cable_filling}
If $|qrs-p|=1$, then for any knot $K$, we have $C_{r,s}(K)(p/q)\cong  K(p/(qs^2))$. \qed
\end{lem}

The genus of a cable knot satisfies $g(C_{r,s}(K))\geq s\cdot g(K)+\frac{(s-1)(|r|-1)}{2}$ \cite[Section 12]{Schubert_Genus}. In particular, a knot $K$ is never isotopic to its cable $C_{r,s}(K)$ for $s\neq \pm 1$. Hence, we obtain the following results, which was already observed by Motegi \cite{Motegi_surgery}.

\begin{prop}\label{thm:cable_slopes}
    There exist infinitely many slopes that are non-strongly-character\-is\-ing for every knot. More concretely, let $(r,s)$ and $(p,q)$ be coprime such that $|qrs-p|=1$ and $s\neq \pm1$, then $p/(qs^2)$ is a non-strongly-characterising slope for any knot $K$.\qed
\end{prop}

\begin{rem}
    Hence, for example, the slope $1/4$ is never strongly-characterising.
\end{rem}

\subsection{Non-cable slopes}
In Proposition~\ref{thm:cable_slopes}, we have constructed infinitely many non-strongly-characterising slopes for every knot. However, all these slopes come from the cable construction. On the other hand, there are a multitude of other constructions of non-strongly-characterising slopes that do not come from cables. In the following, we will discuss some of these constructions. In particular, we will construct knots that have infinitely many non-strongly-characterising slopes of unbounded denominators not coming from the cable construction. Related constructions are given in~\cite{Brakes_non_char,Livingston_surgery,McCoy_Pretzel}.\\

\begin{figure}[htbp]
    \centering
    \includegraphics[width=0.6\linewidth]{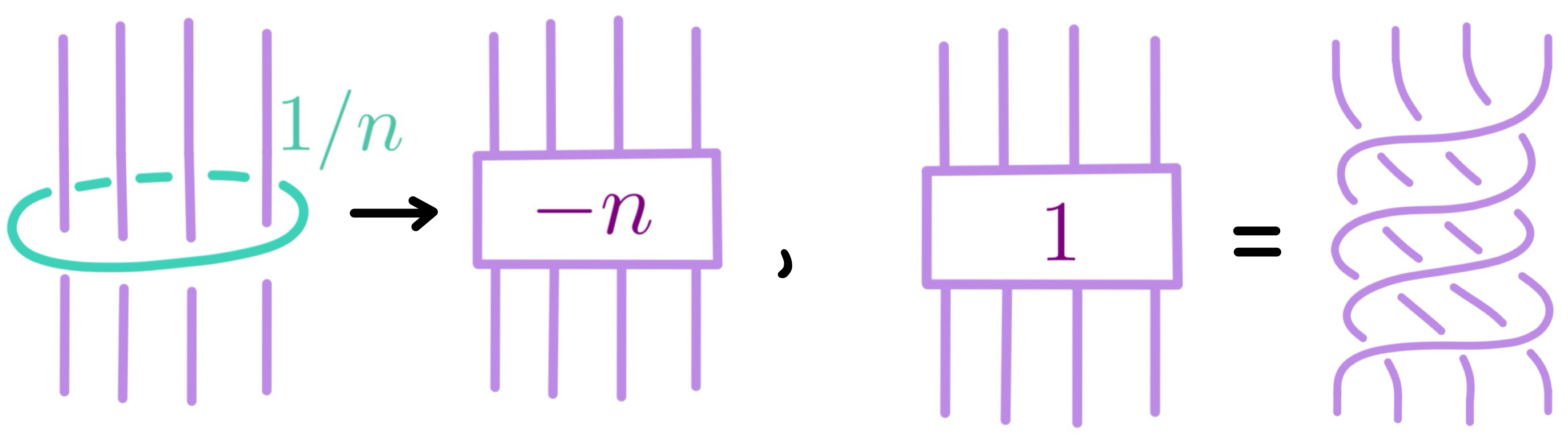} 
    \caption{Sketch of a Rolfson twist (left) and a twist box (right).}
    \label{fig:Rolfson_twist}
\end{figure}

Before we start, let us briefly recall the concept of a Rolfsen twist. Let $L$ be a link with unknotted component $L_0$. For an integer $n\in\Z$, an $(-n)$-\textit{fold Rolfsen }twist along $L_0$ is the process of performing $(1/n)$-surgery on $L_0$. This will remove the component $L_0$ from $L$ and introduce a $(-n)$-fold full twist to all strands that intersect the Seifert disk of $L_0$, as shown in Figure \ref{fig:Rolfson_twist}. To indicate an $n$-fold full twist, we will often draw a twist box as shown in Figure \ref{fig:Rolfson_twist}. \\

All upcoming constructions will be based on the following lemma.

\begin{lem}\label{lem:two_components}
 Let $L=L_1\cup L_2$ be an oriented $2$-component link with unknotted components $L_1$ and $L_2$. Let $l$ be the linking number $\lk(L_1,L_2)$.
 Then for all integers $n,m\in\Z$
 \begin{align*}
  K_m=\,L\left(*,\frac{1}{m}\right) \textrm{ and }
  J_n=\,L\left(\frac{1}{n},*\right)
 	\end{align*}
 are knots in $S^3$, such that $K_m(-ml^2+\frac{1}{n})\cong J_n(-nl^2+\frac{1}{m})$. Here the notation $L(*,\frac{1}{m})$ means that we perform a $(1/n)$-surgery on $L_2$, the second component of $L$ and leave $L_1$, the first component of $L$, unfilled. 
\end{lem}

\begin{proof}
    Since $L_2$ is unknotted, $(1/m)$-surgery on $L_2$ produces again $S^3$ and thus $K_m$ defines a knot in $S^3$. $K_m$ is obtained from $L_1$ by a $(-m)$-fold Rolfsen twist along $L_2$. By symmetry, the same is true for $J_n$. By following the framing change in the Rolfsen twist, we obtain $K_m(-ml^2+\frac{1}{n})\cong L(\frac{1}{n},\frac{1}{m})\cong J_n(-nl^2+\frac{1}{m})$. We refer to Figure~\ref{fig:examples} for an example.
\end{proof}

Note that in Lemma~\ref{lem:two_components} the knots $K_m$ and $J_n$ could in principle be isotopic, and for overly simple or symmetric links $L$ this will indeed be the case. However, there is no reason to expect this to happen in general. We now proceed by describing several cases where $K_m$ and $J_n$ are distinct. 

\begin{defi}
 Consider a link $L=L_1\cup L_2$ with unknotted component $L_2$. Suppose that $L$ is not split and $L_2$ is not the meridian of $L_1$, that is $L_1, L_2$ have geometric linking number at least $2$. We call the set of knots $K_m=L(*, \frac 1m), m\in \mathbb Z$, obtained by twisting $K_0=L_1$ $(-m)$-times along the unknot $L_2$ a {\it twist family}. 
\end{defi}

\begin{thm}[Kouno--Motegi--Shibuya, Theorem 3.2 \cite{Motegi_twisting}]\label{thm:Motegi_twist}
    Let $(K_m)_{m\in\Z}$ be a twist family of knots.
    Then for any integer $m_0$, there exist at most finitely many integers $m$ such that $K_m$ is isotopic to $K_{m_0}$.\qed
\end{thm}

Combining Lemma~\ref{lem:two_components} with Theorem~\ref{thm:Motegi_twist}, yields.

\begin{thm}\label{thm:untwisting_precise}
    Let $K$ be a non-trivial knot that can be unknotted by twisting $m$-times along an unknot $c$ with linking number $l=\lk(K,c)$. Then all but finitely many slopes of the form $-ml^2+\frac{1}{n}$, for $n\in\Z\setminus\{0\}$, are non-strongly-characterising for $K$.    
\end{thm}

\begin{proof}
    Twisting $m$ times along the unknot $c$, yields a $2$-component link consisting of unknots as in Lemma~\ref{lem:two_components}. Thus we get for any integer $n\in\Z$, a knot $J_n$ such that $K(-ml^2+\frac{1}{n})\cong J_n(-nl^2+\frac{1}{m})$. The $(J_n)_{n\in\Z}$ form a twist family of knots. Thus, Theorem~\ref{thm:Motegi_twist} implies that all but finitely many of the $J_n$ are different from $K$. 
\end{proof}

Note that many knots can be untwisted in different ways. In particular, this holds for knots with unknotting number one.

\begin{cor}\label{cor:unknotting_precise}
    Let $K$ be a non-trivial knot that can be unknotted by changing a positive/negative crossing. Then, all but finitely many slopes of the form $\frac 1n$ for $n\neq 0$ and $\mp 4+\frac 1n$ for $n\neq 0$ are non-strongly-characterising for $K$. Here, the sign $\mp$ is $-$ for a positive crossing and $+$ for a negative crossing. 
\end{cor}

\begin{proof}
    A positive/negative crossing can be changed by a full negative or full positive twist along an unknot $c$  circling the crossing. There are two choices of an unknot circling the crossing, one of which has linking number $0$ and the other has linking number $2$ with $K$. See Figure~\ref{fig:examples} for examples. Hence, the result follows from Theorem~\ref{thm:untwisting_precise}.
\end{proof}

Of course, in concrete examples, we expect all (or all but one) of these slopes to be non-strongly-characterising. Indeed, for knots $K$ with untwisting number one, we can improve the statement of Theorem~\ref{thm:untwisting_precise} by distinguishing the $J_n$ from $K$. This is easiest with invariants that behave well under twisting, as for example the hyperbolic volume or the Alexander polynomial, see for example the proof of Theorem~\ref{thm:all_slopes_precise}. 

\begin{ex}\label{ex:trefoil_fig8}
    We discuss the two simplest knots, the trefoil $K3a1$ and the figure eight knot $K4a1$. 
    \begin{enumerate}
        \item First, we note that the figure eight knot is isotopic to its mirror $-K4a1$. Thus it follows that if $K4a1(p/q)\cong J(p/q')$, then 
        $$K4a1(-p/q)\cong (-K4a1)(-p/q)\cong (-J)(-p/q').$$
        In particular, we get that if $p/q$ is non-strongly-characterising for $K$, then $-p/q$ is also non-strongly-characterising for $K$. Thus, for the figure eight knot, it is sufficient to only discuss non-negative slopes.

        In the first two rows of Figure~\ref{fig:examples} we present two different ways to unknot $K4a1$, which yields for every integer $n$ families of knots $J_n$ and $I_n$ such that $K4a1(1/n)\cong J_n(-1)$ and $K4a1(4+1/n)\cong I_n(-1-4n)$. Theorem~\ref{thm:Motegi_twist} implies that all but finitely many of the $I_n$ and $J_n$ are different from $K4a1$. In fact, here the knots $J_n$ are all twist knots and the $I_n$ are double twist knots for which it is straightforward to check (for example using the volumes or the Alexander polynomials) that for positive $n$ the knots $I_n$ and $J_n$ are all different from $K4a1$ and thus any slope of the form $1/n$ or $4+1/n$ for $n\neq 0$ is non-strongly-characterising for the figure eight knot.
        \item The last two rows in Figure~\ref{fig:examples} show the same for the right-handed trefoil knot. There we have constructed two families of knots $G_n$ and $H_n$ such that $K3a1(1/n)\cong G_n(1)$ and $K3a1(4+1/n)\cong H_n(-1-4n)$. The family $G_n$ agrees with the twist knots $J_n$ constructed for the figure eight knot, and the $H_n$ are the $2$-stranded torus knots $T_{2,1-2n}$. Thus it follows that $H_{-1}$ and $G_1$ are both isotopic to $K3a1$ while all the other $H_n$ and $G_n$ are different from $K3a1$. This implies that any slope of the form $1/n$ or $4+1/n$ is non-strongly-characterising for the right-handed trefoil, except $1$ (which is strongly-characterising by~\cite{Ghighini_Poincare_Sphere}) and possibly $3$. 
    \end{enumerate} 
\end{ex}

As communicated to us by Steven Sivek, $3$ is indeed strongly-characterising.

\begin{prop}\label{Sivek_Lemma}  The slopes $1,2,3,4,3/2$, and $4/3$ are strongly characterising for the right-handed trefoil. 
\end{prop}

\begin{proof} 
First, we recall that for a positive slope $p/q$, the surgered manifold $K(p/q)$ is an L-space if and only if $K$ is an L-space knot and $p/q\geq 2g(K)-1$, where $g(K)$ denotes the $3$-genus of $K$ \cite{KMOS_unknot_char}.

Let $p/q\in \{1, 2, 3,4,3/2, 4/3\}$. Since the right-handed trefoil $K3a1$ is an L-space knot with genus one and since $p/q\geq 1= 2g(K3a1)-1$, it follows that $K3a1(p/q)$ is an L-space. If $K$ is another knot such that $K3a1(p/q)\cong K(p/q')$ for some integer $q'$, then $K$ or its mirror $-K$ is also an L-space knot. Since $$4\geq p \geq |p/q'| \geq 2g(K)-1$$ 
it follows that $g(K)\leq 2$. By \cite{Genus_2_L_Space}, the only L-space knots $K$ with $g(K)\leq 2$ are the unknot, the trefoil $K3a1$, and the cinquefoil $K5a2$. 

Next, we observe that since $K3a1$ is the torus knot $T(3,2)$, Lemma~\ref{lemma:torus_knot_filling} below shows that $K3a1(p/q)$ is a Seifert fibered space over $S^2$ with singular fibres of order $2$, $3$, and $6q-p\neq \pm 1$. In particular, $K3a1(p/q)$ is not a lens space (since a lens space is a Seifert fibered space with at most two singular fibres) and thus $K$ cannot be an unknot.

Thus, $|p/q'|\geq 2g(K)-1\geq 1$, which implies $|q'|\leq p\leq 4$. Hence, there are finitely many pairs of knot surgeries left to distinguish. Note that $K5a2$ is the torus knot $T(5,2)$ and thus we can again use Proposition~\ref{lemma:torus_knot_filling} to determine the Seifert fibered structure on these spaces, which reveals them to be different. Indeed, the orders of the fibres of $K3a1(p/q)$ are all positive, while any surgery on $-K3a1$ or $-K5a2$ contains at least one fibre of negative order. It remains to show that $K$ cannot be $K5a2$. For that we notice that any surgery on $K3a1$ contains a fibre of order $3$, while $K5a2(p/q')$, for $|q'|\leq p\leq 4$, never contains a fibre of order $3$.  
\end{proof}


\begin{figure}[htbp]
    \centering
    \includegraphics[width=0.99\linewidth]{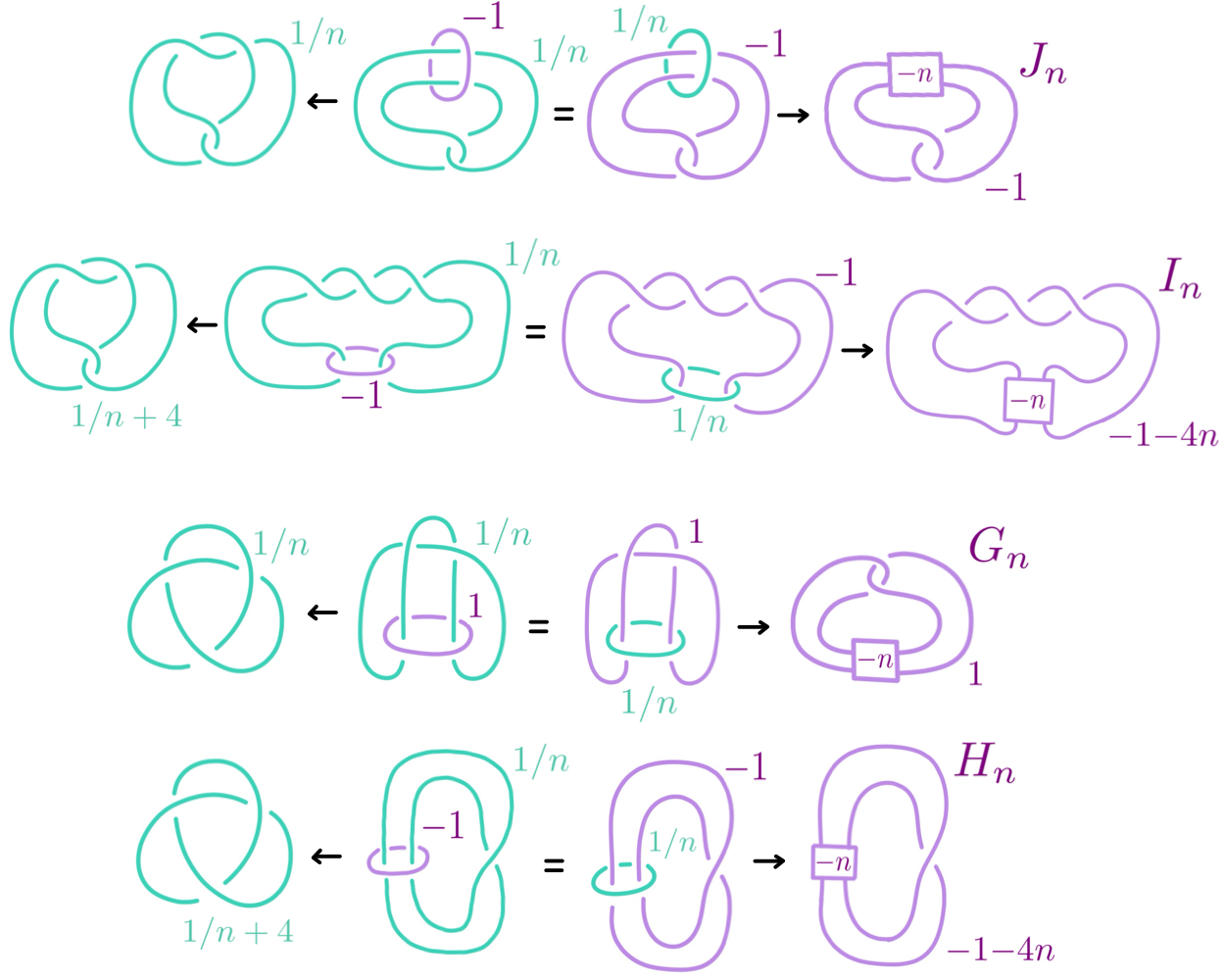} 
    \caption{Non-strongly-characterising slopes of the trefoil and the figure eight knot.}
    \label{fig:examples}
\end{figure}

To generalize the above examples, we need to be able to distinguish knots in a twist family effectively. In~\cite{Baker_Motegi_Alexander}, cf.\ Proposition 2.6 in~\cite{Baker_Motegi_non_char}, Baker--Motegi use the Torres formula~\cite{Torres} to develop the following general method to compute the Alexander polynomial of a twist family of knots.

\begin{lem}[Baker--Motegi]\label{lem:Baker_Motegi_Alexander}
    Let $L$ be an oriented $2$-component link consisting of two unknots $L_1$ and $L_2$ with linking number $\lk(L_1,L_2)=1$ and let $\Delta_L(t_1,t_2)$ be the multi-variable Alexander polynomial of $L$, where the variable $t_i$ is associated to the meridian of the knot $L_i$. Then the Alexander polynomials of the two twist families of knots $K_m=\,L\left(*,\frac{1}{m}\right)$ and $J_n=\,L\left(\frac{1}{n},*\right)$ are given by
\begin{align*}
 \Delta_{K_m}(t)&\,\dot{=}\,\Delta_L(t,t^{m}), \textrm{ and}\\
 \Delta_{J_n}(t)&\,\dot{=}\,\Delta_L(t^{n},t),
\end{align*}
here the notation $\dot{=}$ means that these two polynomials agree up to multiplication with a unit $\pm t^k$. \qed
\end{lem}

With this lemma, we can compute the Alexander polynomials of knots in specific twist families and improve Theorem~\ref{thm:untwisting_precise}.

\begin{thm}\label{thm:all_slopes_precise}
    There exist families of mutually distinct knots $(K_m)_{m\in\Z}$ and $(J_n)_{n\in\Z}$, such that for all non-zero integers $n$ and $m$,
    \begin{itemize}
        \item the knots $K_m$ and $J_n$ are non-isotopic, but 
        \item $K_m(m+\frac{1}{n})$ is orientation-preservingly diffeomorphic to $J_n(n+\frac{1}{m})$.
    \end{itemize}
    In particular, for every integer $m\in\Z\setminus\{0\}$ there exist a knot $K_m$ such that every slope of the form $m+\frac{1}{n}$, for $n\in\Z\setminus\{0\}$, is non-strongly-characterising for $K_m$.
\end{thm}

\begin{proof}
    Let $L=L9a20$ be the $2$-component link shown in Figure~\ref{fig:L9a20}.\footnote{The link $L=L9a20$ was chosen since, in a SnapPy search of the low crossing link census, it was the first link such that $L$ consisted of two unknots of linking number $1$ and such that the associated knots $K_m, J_n$ could be distinguished by the Alexander polynomial for $1<|m|, |n|\leq 20$. The SnapPy computations can be accessed at~\cite{data}.} Since both components are unknots with linking number $l=1$, Lemma~\ref{lem:two_components} implies that for all integers $n,m\in\Z$,
    \begin{align*}
  K_m=\,L\left(*,\frac{1}{m}\right) \textrm{ and }
  J_n=\,L\left(\frac{1}{n},*\right)
 	\end{align*}
 are knots in $S^3$, such that $K_m(-m+\frac{1}{n})\cong J_n(-n+\frac{1}{m})$. To prove the theorem, we will first show that for all integers $n,m\in\Z\setminus\{0\}$ these knots are non-isotopic. For that, we use SnapPy~\cite{SnapPy} to compute the multivariable Alexander polynomial of $L$ to be
\begin{align*}
 \Delta_L(t_1,t_2)= &\,  t_1^2 t_2^4 - 3t_1^2t_2^3 - t_1t_2^4 + 3t_1^2t_2^2 + 4t_1t_2^3 - t_1^2t_2 - 7t_1t_2^2 - t_2^3 \\
 &+ 4t_1t_2 + 3t_2^2 - t_1 - 3t_2 + 1.
\end{align*} 
Lemma~\ref{lem:Baker_Motegi_Alexander} implies that
\begin{align*}
 \Delta_{K_m}(t)\,\dot{=}\,\Delta_L(t,t^m)
 =& \,  t^{4m+2}- 3t^{3m+2}-t^{4m+1} + 3t^{2m+2}
 + 4t^{3m+1}  - t^{m+2}
 \\
 &
 - 7t^{2m+1} - t^{3m} 
 + 4t^{m+1} 
 + 3t^{2m} -t
    - 3t^m + 1\\
    =&~ t^{4m+2}-t^{4m+1}-3t^{3m+2}+4t^{3m+1}-t^{3m}+3t^{2m+2}\\
    & -7t^{2m+1} +3t^{2m}-t^{m+2}+4t^{m+1}-3t^m-t+1,\\
\Delta_{J_n}(t)\,\dot{=}\,\Delta_L(t^n,t)
=&\, t^{2n+4} - 3t^{2n+3}  - t^{n+4} + 3t^{2n+2}  + 4t^{n+3}- t^{2n+1}
  \\
&- 7t^{n+2}  - t^3 + 4t^{n+1}  + 3t^2- t^n
    - 3t + 1\\
    =& ~ t^{2n+4}-3t^{2n+3}+3t^{2n+2}-t^{2n+1}-t^{n+4}+4t^{n+3}\\
    &-7t^{n+2}+4t^{n+1}-t^n-t^3+3t^2-3t+1.
 \end{align*} 
To distinguish all these Alexander polynomials, we first normalise them to have only non-negative exponents and $\Delta(0)>0$. For positive $m$ and $n$ the above formulas are already normalised. For negative $m$ and $n$ the normalised polynomials are:
\begin{align*}
\Delta_{K_m}(t) =& \, -t+3t^{-m+1}+1-3t^{-2m+1}-4t^{-m}+t^{-3m+1} +7t^{-2m}\\ &+t^{-m-1}-4t^{-3m}  -3t^{-2m-1}+t^{-4m}+3t^{-3m-1}-t^{-4m-1} \\
=&~ t^{-4m}-t^{-4m-1}+t^{-3m+1}-4t^{-3m}+3t^{-3m-1}-3t^{-2m+1}\\
&  7t^{-2m}-3t^{-2m-1}+3t^{-m+1}-4t^{-m}+t^{-m-1}-t+1,\\
\Delta_{J_n}(t) = &\, -t^{3} + 3t^{2}  + t^{-n+3} - 3t  - 4t^{-n+2}+ 1+ 7t^{-n+1}  + t^{-2n+2} \\ &- 4t^{-n}  - 3t^{-2n+1}+ t^{-n-1}
    + 3t^{-2n} - t^{-2n-1}\\
    =&~ t^{-2n+2}-3t^{-2n+1}+3t^{-2n}-t^{-2n-1}+t^{-n+3}-4t^{-n+2}\\ 
    &+7t^{-n+1} -4t^{-n}+t^{-n-1}-t^3+3t^2-3t+1.
\end{align*}
 Note that for $|m|\geq 3$ and for $n\geq 4$ or $n\leq -5$, the latter normalised expressions we give for each of these polynomials have monotonically decreasing exponents. In particular, for $|m|\geq 3$, the coefficient of the term of the normalised $\Delta_{K_m}(t)$ of the second largest exponent is $-1$. On the other hand, for $n\geq 4$ or $n\leq -5$, the coefficient of the term of the normalised $\Delta_{J_n}(t)$ is $-3$. Hence, in this case, the knots $K_m$ and $J_n$ are distinct.
 
 If $|m|<3$, then some exponents in the above expression for $\Delta_{K_m}(t)$ agree and hence $\Delta_{K_m}(t)$ has fewer than 13 distinct exponents. Hence $K_m$ with $|m|<3$ and $J_n$ with $n\geq 4$ or $n\leq -5$ are distinct. Similarly $K_m$ with $|m|\geq 4$ and $J_n$ with $-5<n<4$ are distinct. 
 
 Hence, it is sufficient to show that the above normalised expressions for $\Delta_{K_m}(t)$ and for $\Delta_{J_n}(t)$ with $|m|<3$ and $-5<n<4$ are distinct. These are finitely many pairs which can be distinguished by hand. Indeed, the Alexander polynomials for these 13 knots are all distinct apart from $\Delta_{J_1}(t)\dot=\Delta_{K_1}(t)$ and $\Delta_{J_{-1}}(t)\dot=\Delta_{K_{-1}}(t).$\footnote{Note that this holds true for any link $L$ with linking number $1$ by Lemma~\ref{lem:Baker_Motegi_Alexander}.}
 
 Finally, we may use SnapPy~\cite{SnapPy} to check that the four knots $J_{-1}$, $K_{-1}$, $J_1$, and $K_1$ have different volumes and are therefore non-isotopic. In fact, $J_1$ is the $10$-crossing knot $K10n23$, $J_{-1}$ the 11-crossing knot $K11n86$, while experiments with SnapPy suggest that $K_1$ has crossing number $19$, and $K_{-1}$ crossing number $22$. These computations can be accessed at~\cite{data}. 
\end{proof}

\begin{figure}[htbp]
    \centering
    \includegraphics[width=0.99\linewidth]{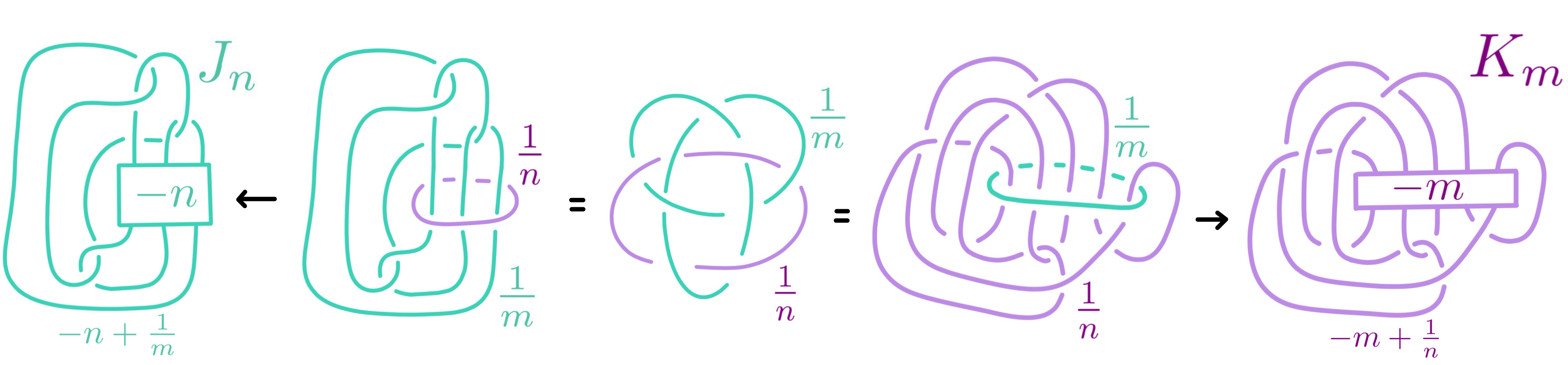} 
    \caption{Three diagrams for the link $L9a20$ (middle), the families $J_n$ (left), and $K_m$ (right).}
    \label{fig:L9a20}
\end{figure}

\begin{rem}
    Note that in Theorem~\ref{thm:all_slopes_precise} we assume $m,n\neq 0$. However, in Example~\ref{ex:trefoil_fig8} we have shown that every reciprocal integer slope is non-strongly-characterising for the figure eight knot. Moreover, there are strong results about non-characterising slopes of the form $1/n$. Using surgeries on Brunnian links, for every pair of integers $n\in\Z$ and $k\in\N$ Livingston~\cite{Livingston_surgery} has constructed  pairwise non-isotopic knots $K_1,\ldots,K_k$ that all have diffeomorphic $(1/n)$-surgeries. 
\end{rem}

\section{Strongly-characterising slopes - the hyperbolic case}\label{section:hyperbolic_knots}
In this section, we prove that all hyperbolic knots have infinitely many strongly-characterising slopes. The more general result that any non-trivial knot has infinitely many strongly-characterising slopes is proven in Section~\ref{section:general_knots}. However, the proof in the hyperbolic case is particularly short and mostly self-contained, motivating its independent and more elementary exposition. Let us review the two main tools used in the upcoming proof, namely hyperbolic geometry and linking forms of rational homology spheres.

\subsection{Hyperbolic 3-manifolds} A 3-manifold is {\it hyperbolic} if its interior admits a complete, finite volume Riemannian metric of constant sectional curvatures $-1$. For an introduction to hyperbolic 3-manifolds, see \cite{Martelli_book}. A knot $K\subset S^3$ is {\it hyperbolic} if $S^3-K$ is hyperbolic. Let us give an overview of some hyperbolic geometry used to study strongly-characterising slopes. First of all, the Mostow rigidity theorem \cite{Mostow_rigidity} states that a hyperbolic 3-manifold has a unique hyperbolic metric up to isometry. It is a theorem of Thurston \cite{Thurston_Dehn_Surgery} that being hyperbolic is preserved under sufficiently complicated Dehn fillings. Moreover, for $K$ hyperbolic and $|q|$ sufficiently large, $K(p/q)$ is hyperbolic and the core curve of the Dehn filling is the unique shortest geodesic in $K(p/q)$. Hence, if $K, K'$ are hyperbolic, $K(p/q)\cong  K'(p/q')$, and $|q|, |q'|$ are sufficiently large, then the core curves of the Dehn fillings are both the unique shortest geodesic and thus equivalent. This in turn, implies that $S^3-K\cong  S^3-K'$. We will apply the following result of Lackenby. 

\begin{thm}[Lackenby, Theorem 1.3 \cite{Lackenby_char}] \label{thm:lackenby_char_slope}  Given a hyperbolic knot $K\subset S^3$, there is a constant $C(K)$ with the following property: If there is an orientation-preserving diffeomorphism between $K(p/q)$ and $K'(p'/q')$ for some hyperbolic knot $K'\subset S^3$ and $|q'|\geq C(K)$, then there $K=K'$ and $p/q=p'/q'$. \qed
\end{thm}

We also need the following well-known result, a proof of which can be found in Section~\ref{section:general_knots}. 

\begin{lem}\label{lemma:hyp_reduction} If $K'$ is a non-hyperbolic knot with $K'(p/q')$ hyperbolic for $|q'|>2$, then there is a hyperbolic knot $K''$ and $w>1$ such that $K''(p/(q'\cdot w^2))\cong  K'(p/q')$. 
\end{lem}

\subsection{Linking forms}

The second tool we use is the linking form.  

\begin{defi}
    Consider a closed, rational homology $3$-sphere $M$, the {\it linking form}$$\lambda_M\colon H_1(M; \mathbb Z)\times H_1(M; \mathbb Z)\to \mathbb Q/\mathbb Z$$ is defined as follows. For $a,b\in H_1(M; \mathbb Z)$ take $\alpha, \beta\in C_1(M; \mathbb Z)$ with $[\alpha]=a, [\beta]=b$. Since $H_1(M; \mathbb Z)$ is finite, there exists $N\in \mathbb Z$ and $\sigma\in C_2(M; \mathbb Z)$ such that $N\cdot \alpha=\partial \sigma$. Then we define $$\lambda_M(a,b):=\frac{\langle \sigma, \beta\rangle}{N}$$ where $\langle \sigma, \beta\rangle\in \mathbb Z$ denotes the intersection form of $M$.
    \end{defi}

In this definition $\lambda_M$ is well-defined as the value of $\langle \sigma, \beta\rangle/N\in\Q/\Z$ is independent of the choices of $\alpha, \beta, N$, and $\sigma$.\\

Recall that for an oriented knot $K\subset S^3$, the Dehn filling $K(p/q)$ has homology $H_1(K(p/q); \mathbb Z)$ canonically isomorphic to $\mathbb Z/p\mathbb Z$ with generator $1+p\mathbb Z$ given by the oriented meridian. Then the linking form on $K(p/q)$ is computed as follows.

\begin{prop}\label{prop:linking_form_formula}
    Given an oriented knot $K\subset S^3$ and a slope $p/q\in \mathbb Q\setminus\{0\}$, the linking form on $K(p/q)$ satisfies $\lambda_{K(p/q)}(a,b)=-\frac qp \cdot a\cdot b$ under the canonical isomorphism $H_1(K(p/q); \mathbb Z)\cong\mathbb Z/p\mathbb Z$.
\end{prop}

\begin{proof}
Let $\mu, \lambda\subset \partial N(K)$ denote the meridian and the longitude of $K$ (oriented by the standard right-hand-rule convention). 
Then the slope $p\mu+q\lambda$ bounds a disc in $K(p/q)$ and $\lambda$ is the boundary of the Seifert surface $\Sigma_K\subset S^3 \setminus N(K)$. Hence $p\cdot \mu=-q\cdot \partial \Sigma_K\in C_1(K(p/q); \mathbb Z).$ By definition of the linking form, we get 
\begin{equation*}
     \lambda_M([\mu], [\mu])=-\frac{q}{p}\langle \mu, \Sigma_K\rangle=-\frac{q}{p}.\qedhere
\end{equation*}
\end{proof}

\begin{rem}
One can equivalently define the linking form on rational homology 3-spheres via the cap product and Poincar\'e duality. A careful exposition of this appears in \cite{Linking_forms_alg_top}. Let us briefly mention this alternative definition.
$\lambda_M(a,b)=\Omega(a)(b)$ for $\Omega$ the composition: 
$$H_1(M; \mathbb Z)\xrightarrow[]{PD}H^2(M; \mathbb Z)\xrightarrow[]{\beta^{-1}}H^1(M; \mathbb Q/\mathbb Z)\xrightarrow[]{\varphi\mapsto \langle \varphi, -\rangle_{Kro}} \text{Hom}(H_1(M; \mathbb Z); \mathbb Q/\mathbb Z).$$
Here $PD$ denotes Poincar\'e duality, $\beta$ the Bockstein homomorphism, and $\langle ~ , ~\rangle_{Kro}$ the Kronecker pairing. 
\end{rem}

We summarise the discussion on linking forms with the following result.

\begin{prop} \label{prop:linking_framing} Let $K, K'\subset S^3$ be knots with slopes $p/q, p'/q'\in \mathbb Q$, for $p,p'>0$. If $K(p/q)\cong  K'(p'/q')$ then $p=p'$ and $q\cdot q'\equiv n^2 \pmod p$ for some $n\in \mathbb Z$. 
\end{prop}

\begin{proof}
An orientation preserving diffeomorphism $f\colon K(p/q)\to K'(p'/q')$ induces an isomorphism $f_*\colon H_1(K(p/q); \mathbb Z)\to H_1(K'(p'/q'); \mathbb Z)$. Hence $p=p'$.\\
Moreover, the linking forms must satisfy $\lambda_{K(p/q)}(a,b)=\lambda_{K'(p/q')}(f_*(a), f_*(b))$. Under the identifications of $H_1(K(p/q); \mathbb Z)$ and $H_1(K'(p/q'); \mathbb Z)$ with $\Z/p\Z$ given by the two respective Dehn fillings, we know that 
\begin{itemize} 
\item[-] $f_*$ has the form $x \pmod p \mapsto r\cdot x \pmod p$ for some $r\in \mathbb Z$, 
\item[-] $\lambda_{K(p/q)}(a,b)=-\frac qp\cdot a\cdot b$ and $\lambda_{K'(p/q')}(a,b)=-\frac{q'}p\cdot a\cdot b$.
\end{itemize}
Hence, for some $r\in \mathbb Z$ the maps $\mathbb Z/p\mathbb Z\times \mathbb Z/p\mathbb Z\to \mathbb Q/\mathbb Z, (a,b)\mapsto  -\frac{q}{p}\cdot a\cdot b$ and $\mathbb Z/p\mathbb Z\times \mathbb Z/p\mathbb Z\to \mathbb Q/\mathbb Z, (a,b)\mapsto -\frac{q'}{p}\cdot (r\cdot a)\cdot (r\cdot b)$ are equal. In particular, this implies $q\equiv  q'\cdot r^2 \pmod p$ and hence $q\cdot q'\equiv (q'\cdot r)^2 \pmod p$. 
\end{proof}

\subsection{Quadratic reciprocity} In light of the previous theorem, let us recall some results on quadratic reciprocity. The classical reference for this topic is \cite{Gauss_original}. Given a prime $p$ and $q$ not a multiple of $p$, we say $q$ is a {\it (quadratic) residue mod} $p$ if $q\equiv n^2 \pmod p$ for some $n\in \mathbb Z$. For an odd prime $p$, $(p-1)/2$ of the numbers $1, \ldots, p-1$ are residues and the remaining $(p-1)/2$ are non-residues. A product of residues is clearly again a residue mod $p$, and hence the product of a residue and a non-residue is a non-residue. Let us recall the {\it Legendre symbol} $$\left(\frac{q}{p}\right):=\begin{cases} 1  &\text{if }q\text{ is a residue mod }p,\\ -1  &\text{if }q\text{ is a non-residue mod } p.\end{cases}$$
The law of quadratic reciprocity states that for distinct odd primes $p_1, p_2$ we have $$\left(\frac{p_1}{p_2}\right)\left(\frac{p_2}{p_1}\right)=(-1)^{\frac{p_1-1}{2}\frac{p_2-1}{2}}.$$
Finally, it is known that $-1$ is a quadratic residue mod an odd prime $p$ if and only if $p\equiv 1 \pmod 4$ and $2$ is a quadratic residue mod an odd prime $p$ if and only if $p\equiv \pm 1 \pmod 8$.\\

We are now ready to prove the main result of this section. 

\subsection{Existence of strongly-characterising slopes}
\begin{thm}  Any hyperbolic knot $K\subset S^3$ has infinitely many strongly-character\-is\-ing slopes.
\end{thm}

\begin{proof}
Consider some slope $p/q\in \mathbb Q$ such that $K(p/q)$ is orientation-preservingly diffeomorphic to $K'(p/q')$ for some knot $K'\subset S^3$.\\

\noindent
Take $p$ large enough so that $K(p/q)$ is hyperbolic. \\
\textbf{Case 1.} $K'$ is hyperbolic: By Theorem 3.1, there is a constant $C(K)$, depending only on $K$, such that if $|q'|\geq C(K)$, then $K=K'$.\\
\textbf{Case 2.} $K'$ is not hyperbolic: If $|q'|>2$, then by Lemma \ref{lemma:hyp_reduction}, there is a hyperbolic knot $K''$ and $w'>1$ such that $K(p/q)\cong  K''(p/(q'\cdot w^2))$. Again, by Theorem~\ref{thm:lackenby_char_slope}, there is a constant $C(K)$, depending only on $K$, such that if $|q'w^2|\geq C(K)$ then $K=K''$ and $p/q=p/(q'w^2)$. In particular, $q$ has repeated prime factors. \\

\noindent
In either case, if $q$ has no repeated prime factors and if $p$ is sufficiently large and $|q'|>C(K)$, then $K=K'$. By Proposition~\ref{prop:linking_framing}, we know that $q\cdot q'$ is a quadratic residue mod $p$. Suppose additionally that $p\equiv 1 \pmod 4$ is prime and that $q$ is a quadratic non-residue mod $p$. Then $q'$ is a quadratic non-residue mod $p$. Since $-1$ is a residue mod $p$, we know that $|q'|$ is a quadratic non-residue mod $p$. Hence, if $p\equiv 1 \pmod 4$ is sufficiently large, $q$ has no repeated prime factors and is a non-residue mod $p$, and integers between $1$ and $C(K)$ are quadratic residues, then $K=K'$ and $p/q$ is strongly-characterising. So, to find infinitely many strongly-characterising slopes $p/q$ for $K$, it is sufficient to find infinitely many primes $p$ such that
\begin{itemize}
\item[-] $p\equiv 1 \pmod 4$; 
\item[-] All integers between $1$ and $C(K)$ are quadratic residues mod $p$; 
\item[-] Some $q$ with $\gcd(p,q)=1$ is a non-residue mod $p$ and has no repeated prime factors. 
\end{itemize}
Recall that $2$ is a residue mod $p$ if and only if $p\equiv \pm 1 \pmod 8$. Let $p_1, \ldots, p_n$ be all the odd primes less than $C(K)$ and let $q$ be some other odd prime. Since the product of quadratic residues are quadratic residues, it is sufficient to find infinitely many primes $p\equiv 1 \pmod 8$ such that for all $i=1, \ldots, n$, $p_i$ is a residue mod $p$ and $q$ is a non-residue mod $p$. If $\gcd(p_i,p)=1$ and $p\equiv 1\pmod 4$ then, by the law of quadratic reciprocity we have
$$\left(\frac{p_1}{p}\right)\left(\frac{p}{p_1}\right)=(-1)^{\frac{p_1-1}{2}\frac{p_2-1}{2}}=1.$$
Hence, $p_1$ is a residue mod $p$ if $p$ is a residue mod $p_1$, and this in turn is the case if $p\equiv 1\pmod {p_1}$. In summary, it is sufficient to find infinitely many primes $p>C(K)$ such that
\begin{itemize}
\item[-] $p\equiv 1\pmod 8$; 
\item[-] $p\equiv 1\pmod {p_i}$ for $i=1, \ldots, n$; 
\item[-] $p\equiv r \pmod q$ for $r\in \mathbb Z$ some non-residue mod $q$.
\end{itemize}
By the Chinese remainder theorem, there is some $a\in \mathbb Z$ with $\gcd(a, 8\cdot p_1\cdot\ldots\cdot p_n\cdot q)=1$ such that these three conditions are equivalent to $p\equiv a \pmod{8\cdot p_1\cdot\ldots\cdot p_n\cdot q}$. This completes the proof, since Dirichlet's theorem on arithmetic progressions states that for positive coprime integers $a,d$ there are infinitely many primes $p\equiv a\pmod d$. 
\end{proof}

\section{Strongly-characterising slopes - the general case}\label{section:general_knots}
In this section, we prove that any knot has infinitely many strongly-characterising slopes, by using the ideas from Section~\ref{section:hyperbolic_knots} and a more careful analysis of JSJ-decompositions. Let us recall the definitions of a JSJ-decompositions and related concepts.

\subsection{JSJ-decompositions of knots}

\begin{defi}
    A {\it JSJ-decomposition} of a 3-manifold $M$ is a minimal collection $\mathcal{T}$ of embedded incompressible, boundary-incompressible tori in $M$ such that $\mathcal{T}$ cuts $M$ into hyperbolic and Seifert fibred pieces. 
\end{defi}

The Geometrization theorem \cite{Perelman1,Perelman2,Thurston,JS,J}  states that any compact, irreducible 3-manifold with toroidal boundary admits a JSJ-decomposition, which is unique up to isotopy.\\

For a definition and introduction to Seifert fibred spaces, consult \cite{Martelli_book}. In short, a {\it Seifert fibration} of a 3-manifold $M$ is a well-behaved partition of $M$ into $S^1$'s, called {\it fibres}, which away from a finite set of, so-called {\it singular fibres}, is a trivial $S^1$-bundle over a surface. There is a well-defined $p\in \mathbb N$ associated to each singular fibre, which we call the {\it order}. Finally, quotienting each fibre of $M$ to a point, we will obtain a space homeomorphic to a surface $\Sigma$, and we say $M$ has a Seifert fibration over $\Sigma$. In many situations, Seifert fibrations are unique. 

\begin{defi}
    Say two Seifert fibrations of a manifold $M$ are {\it isomorphic} if there is a diffeomorphism $\phi\colon M\to M$ mapping fibres of one Seifert fibrations to fibres of the other. Say two Seifert fibrations of $M$ are {\it isotopic} if there exists such an isomorphism $\phi\colon M\to M$ isotopic to $id$. 
\end{defi}

\begin{thm} [Waldhausen, Theorem  10.1 \cite{Waldhausen_SFS}] \label{thm:Seifert_unique}  A Seifert fibred 3-manifold $M$ admits a unique Seifert fibration over an oriented surface up to isomorphism unless $M$ is a lens space or solid torus. Moreover, all Seifert fibrations of a solid torus have at most one singular fibre, and all Seifert fibrations of a lens space over an oriented surface have at most two singular fibres.\\
Finally, if $M$ has non-empty boundary, then $M$ has a unique Seifert fibration over an oriented surface up to isotopy unless $M$ is a solid torus or $S^1\times S^1\times [0,1]$. \qed
\end{thm}

Let us mention an important example of a Seifert fibred space.

\begin{ex}\label{ex:SFS_cable}
    Consider the partition of the torus $S^1\times S^1$ into circles of slopes $(s,r)$ with $s,r$ coprime. Taking a product with $(0, 1]$ we can extend this to a $S^1$-bundle on $S^1\times (D^2-\{0\})$. Adding in the circle $S^1\times \{0\}\subset S^1\times D^2$ we obtain a Seifert fibration of $S^1\times D^2$ with one singular fibre of order $s$. Remove one non-singular fibre to construct $C_{r,s}\subset S^1\times D^2$ which we call a {\it cable space}. In other words, $C_{r,s}$ is the complement of the torus knot $T_{r,s}\subset S^1\times D^2$. From our construction, we see that $C_{r,s}$ admits a Seifert fibration over the annulus with one singular fibre of order $s$. Finally, note that in $\partial(S^1\times D^2)$ equipped with the standard meridian, longitude framing a fibre of the Seifert fibration of $C_{r,s}$ has slope $r/s$. 
\end{ex}

\begin{prop} \label{lemma:torus_knot_filling} (Moser \cite{Moser}, cf. Theorem 7.2 in~\cite{BKM_QA}) The Dehn filling $T_{a,b}(p/q)$ of the torus knot $T_{a,b}$ is diffeomorphic to 
\begin{enumerate}
    \item a Seifert fibred space over $S^2$ with singular fibres of orders $a,b$, and $abq-p$, if $abq-p\neq 0$, 
    \item $L(a,b)\#L(b,a)$, if $p=qab$. \qed
\end{enumerate}
\end{prop}

We will take advantage of the fact that, by work of Budney, JSJ-decompositions of knot complements are particularly well-understood, as is detailed in \cite{Budney_knot_JSJ}. However, first let us make the following simple observation. Any JSJ-torus $T\subset S^3 \setminus N(K)\subset S^3$ of $S^3 \setminus N(K)$ is compressible in $S^3$ but not compressible in $S^3 \setminus N(K)$. Hence, $T$ bounds a knot complement $S^3 \setminus N(J)\subset S^3 \setminus N(K)$ on one side and a pattern $V_P\subset S^3 \setminus N(K)$ on the other. In summary, $T$ decomposes $K$ as a satellite knot with pattern $P$ and companion $J$. 

\begin{thm}[Budney, Theorem 4.18 \cite{Budney_knot_JSJ}] \label{thm:Budney_knot_JSJ}  In the JSJ-decomposition of $S^3 \setminus N(K)$ for a non-trivial knot $K$, each $JSJ$-component is one of the following:
\begin{enumerate}
    \item the complement of a torus knot in $S^3$, 
    \item a {\it composing space}, that is a Seifert fibred space over a planar surface with at least 3 boundary components and without singular fibres, 
    \item a hyperbolic 3-manifold, 
    \item a {\it cable space}, that is a Seifert fibred space over an annulus with one singular fibre. \qed
\end{enumerate}
\end{thm}

\subsection{JSJ-decompositions under Dehn filling}

Work of Sorya \cite{Sorya_satellite_char} additionally gives a detailed account of how JSJ-decompositions of knots behave under Dehn filling. Given a JSJ-decomposition of a knot complement, call the JSJ-component containing the boundary {\it outermost}.

\begin{thm}[Sorya, Proposition 3.5 \cite{Sorya_satellite_char}]\label{thm:knot_JSJ_filling_component} Consider a non-trivial knot $K\subset S^3$ and suppose $|q|>2$. Denote the JSJ-decomposition of $S^3 \setminus N(K)$ by $Y_0\cup \ldots \cup Y_k$, with $Y_0$ the outermost component. Then the Dehn filling $Y_0(p/q)$ along the boundary torus of $S^3 \setminus N(K)$ satisfies: 
\begin{enumerate}
    \item if $Y_0$ is hyperbolic and $k\geq 1$, then $Y_0(p/q)$ is hyperbolic, 
    \item if $Y_0$ is a hyperbolic knot complement and $|q|>8$, then $Y_0(p/q)$ is hyperbolic,
    \item if $Y_0$ is a composing space, then $Y_0(p/q)$ is a Seifert fibred space over a planar surface with at least two boundary components and a singular fibre of order $|q|$. Moreover, the singular fibre is the core curve of the Dehn filling, 
    \item if $Y_0$ is an $(r,s)$-cable space and $|qrs-p|>1$, then $Y_0(p/q)$ is a Seifert fibred space over the disc with two singular fibres of orders $|qrs-p|$ and $s$. Moreover, the singular fibre of order $|qrs-p|$ is the core curve of the Dehn filling, 
    \item if $Y_0$ is an $(r,s)$ cable space and $|qrs-p|=1$, then $Y_0(p/q)$ is a solid torus. \qed
\end{enumerate}
\end{thm}

\begin{thm}[Sorya, Proposition 3.6 \cite{Sorya_satellite_char}]\label{thm:JSJ_filling} Consider a non-trivial knot $K\subset S^3$ and suppose $|q|>2$. Denote the JSJ-decomposition of $S^3 \setminus N(K)$ by $Y_0\cup \ldots \cup Y_k$, with $Y_0$ the outermost component. Then the JSJ-decomposition of $K(p/q)$ is either 
$$Y_0(p/q)\cup Y_1\cup \ldots Y_k, \text{ or}$$
$$Y_1(p/(qs^2))\cup Y_2\cup \ldots \cup Y_k,$$
where in the first case we fill $Y_0$ along the boundary torus of $S^3 \setminus N(K)$ and in the second case we fill $Y_1$ along the torus $Y_0\cap Y_1$. Moreover, the second case occurs precisely when $K$ is a cable knot $C_{r,s}(J)$, $Y_1$ is the outermost JSJ-component of $S^3-J$, and $|qrs-p|=1$. \qed
\end{thm}

In light of this theorem, the following definition is well-defined.

\begin{defi} Consider a knot $K\subset S^3$ and suppose $|q|>2$. The {\it surgered piece} of $K(p/q)$ is the JSJ-component containing the core curve of the Dehn filling. In the notation of Theorem~\ref{thm:JSJ_filling}, the surgered piece is $Y_0(p/q)$ in the first case and $Y_1(p/qs^2)$ in the second case.
\end{defi}

The following result is the key technical input from \cite{Sorya_satellite_char}. 

\begin{thm}\label{thm:sorya_main}
Given a knot $K\subset S^3$ there exists a constant $C_{\text{JSJ}}(K)>2$, depending only on $K$, such that if $K(p/q)\cong  K'(p/q')$ for another knot $K'\subset S^3$ and $|q|, |q'|\geq C_{\text{JSJ}}(K)$, then the diffeomorphism sends the surgered piece of $K(p/q)$ to the surgered piece of $K'(p/q')$. 
\end{thm}
\begin{proof}
    For $q=q'$ this is exactly \cite[Proposition 1.2]{Sorya_satellite_char}. We obtain a proof of this theorem by verbatim repeating the proof of \cite[Proposition 1.2]{Sorya_satellite_char} and replacing $q$ by $q'$ every time $p/q$ refers to a slope in the boundary of $S^3 \setminus N(K')$ or in the boundary $\mathcal{P}'$ of $V_{P'}$. In particular, we make the following replacements
    \begin{itemize}
\item[-] replace $S^3_{K'}(p/q)$ by $S^3_{K'}(p/q')$, 
\item[-] replace $V_{P'}(\mathcal{P}'; p/q)$ by $V_{P'}. (\mathcal{P}'; p/q')$,
\item[-] replace the  expression $q(w')^2\mu_J+y\lambda_J$ for the meridian of $V_P(\mathcal{P}; p/q)$ in \cite[Lemma 4.3]{Sorya_satellite_char} by the expression $q'(w')^2\mu_J+y\lambda_J$.\qedhere
\end{itemize}
\end{proof}

Let us briefly record several results that follow from Theorems~\ref{thm:knot_JSJ_filling_component} and \ref{thm:JSJ_filling}. 

\begin{lem} \label{lemma:non-cable}
Let $K\subset S^3$ be a non-trivial knot and suppose that $|q|, |q'|>2$ and that $|q|$ is a prime. Then $K(p/q)$ is not diffeomorphic to a Dehn filling $K'(p/q')$ for $K'=C_{r,s}(K)$ a non-trivial cable of $K$. 
\end{lem}
\begin{proof}
    Suppose $K(p/q)\cong  K'(p/q')$. By Theorem~\ref{thm:JSJ_filling} we must have $|q'r s-p|=1$, as otherwise $K(p/q)$ will have one fewer JSJ-component than $K'(p/q')$. If $|q'rs-p|=1$, then $K(p/q)\cong  K'(p/q')\cong  K(p/q's^2)$. Hence, by \cite[Theorem 1.2]{Ni_Cosmetic}, we must have $q=\pm q's^2$, contradicting the primality of $q$.
\end{proof}

\begin{lemma:hyp_reduction}
Suppose $K'$ is a non-hyperbolic knot with $K'(p/q')$ hyperbolic for $|q'|>2$. Then there is a hyperbolic knot $K''$ with $K''(p/(q'\cdot w^2))\cong  K'(p/q')$ for some $w>1$.  
\end{lemma:hyp_reduction}

\begin{proof}
    Let $K'$ have JSJ-decomposition $Y'_0\cup \ldots \cup Y'_k$. If $|q'|>2$, then by Theorem~\ref{thm:JSJ_filling} $K'(p/q')$ has JSJ-decomposition $Y'_0(p/q')\cup Y'_1\cup \ldots \cup Y'_k$ or $Y'_1(p/(q's^2))\cup Y'_2\cup \ldots \cup Y'_k$. Moreover, in the latter case, $K'=C_{r,s}(K'')$ and $|qrs-p|=1$. In the former case, for $K'(p/q')$ to be hyperbolic, we must have $k=0$ and by Theorem~\ref{thm:knot_JSJ_filling_component} $Y'_0$ is a hyperbolic knot complement. In the latter case, for $K'(p/q')$ to be hyperbolic, we must have $k=1$ and $Y'_1\cong  S^3\setminus K''$ must be hyperbolic. Finally, observe that in the latter case, since $|qrs-p|=1$ and $K'=C_{r,s}(K'')$ by Lemma~\ref{lemma:cable_filling} $K'(p/q')\cong  K''(p/(q's^2))$.
\end{proof}

\subsection{Hyperbolic Dehn fillings} Finally, let us briefly introduce some results to tackle Dehn fillings of hyperbolic manifolds. These techniques are based on the angle-deformation techniques of Hodgeson and Kerckhoff \cite{HK_Angle_Annals}, and we will also be referencing the work of Futer, Purcell, and Schleimer \cite{Futer_Purcell_Schleimer_bound}.\\

Consider a hyperbolic 3-manifold $M$ with torus boundary $T\subset \partial M$. By considering a maximal horoball about $T$, $T$ inherits a canonical Euclidean metric. Given a slope $r\in H_1(T;\Z)$, we may consider two quantities: the {\it length} $l(r)$ of a geodesic representative of $r$ with respect to this Euclidean metric and the {\it normalised length} $\hat{L}(r):=l(r)/\sqrt{\text{Area}(T)}$. If $l(r)$ or $\hat{L}(r)$ is large, we know that the Dehn filling $M(r)$ of $M$ along $r$ is hyperbolic and, in fact, we obtain information about the hyperbolic geometry of $M(r)$. To be precise, we have the following.

\begin{thm} \label{thm:Dehn_surgery_bounds} Consider a hyperbolic 3-manifold $M$ and a slope $r$ in a boundary torus of $\partial M$. If $\hat{L}(r)\geq 10.69$, then 
\begin{itemize}
\item[-] $M(r)$ is hyperbolic, 
\item[-] the core curve $\gamma$ of the Dehn filling $M(r)$ is a geodesic in $M(r)$ of length $l(\gamma)<2\pi/(\hat{L}^2-28.78)$, 
\item[-] if the shortest geodesic of $M$ has length $\text{sys}(M)$, then every geodesic in $M(r)$ that is not $\gamma$ has length at least $\min\{0.0735, 0.5052 \cdot\text{sys}(M)\}.$
\end{itemize}
\end{thm}

\begin{proof} By \cite[Theorem 1.3]{HK_Angle_Annals} if $\hat{L}(r)\geq 7.515$ then the Dehn filling $M(r)$ will be hyperbolic. Moreover, \cite{HK_Angle_Annals} proves the existence of a 1-parameter family of cone-manifolds $M_\alpha, \alpha\in [0, 2\pi]$, with $M_0=M$, $M_{2\pi}=M(r)$, and with $M_\alpha, \alpha>0$ homeomorphic to $M(r)$ with geodesic singular locus of cone-angle $\alpha$ along the core curve of the Dehn filling $M(r)$. For an introduction to cone-manifolds see \cite{Futer_Purcell_Schleimer_bound}, however in this proof the only property of cone-manifolds $M_\alpha$ we rely on is that in the (non-singular) hyperbolic manifold $M(r)=M_{2\pi}$ the core curve $\gamma$ of the Dehn filling $ M(r)$ is a geodesic. It is additionally shown in \cite[Corollary 6.13]{Futer_Purcell_Schleimer_bound} that, if $\hat{L}(r)\geq 7.823$, then $l(\gamma)<2\pi/(\hat{L}(r)^2-28.78)\leq 0.0735$. 

Finally, suppose $\delta\subset M(r)$ is some other geodesic of length $l_{M(r)}(\delta)<0.0735$. By our previous discussion $l(\gamma)=l_{M(r)}(\gamma)<0.0735$. Hence, by \cite[Corollary 7.20]{Futer_Purcell_Schleimer_bound}, we know that $\delta$ is a geodesic in $M(r)\setminus\gamma=M$ of length $l_{M}(\delta)\leq 1.9793\cdot  l_{M(r)}(\delta)$. So $l_{M(r)}(\delta)\geq 0.5052 \cdot\text{sys}(M)$.
\end{proof}

For hyperbolic links in $S^3$, filling slopes $p/q$ with large $|q|$, have large normalised length. Let us make this statement precise.

\begin{lem} \label{lemma:length_bound} Let $M$ be a hyperbolic 3-manifold $M$ with a framing of one of its boundary tori $T\subset \partial M$ such that the corresponding Dehn filling $M(1/0)$ is not hyperbolic. Then the normalised length of the slope corresponding to $p/q$ satisfies $\hat{L}(p/q)\geq |q|/5$. 
\end{lem}

\begin{proof}
Since $M(1/0)$ is not hyperbolic, the 6-theorem \cite{Agol-6-Thm, Lackenby-6-Thm} tells us that $l(1/0)\leq 6$. Moreover, by \cite[Lemma 2.1]{Cooper_Lackenby_Dehn_Surgery}, we know that for two slopes $s_1, s_2$ on $T$ we have $l(s_1)\cdot l(s_2)\geq \sqrt{3}\cdot \Delta(s_1, s_2)$, where $\Delta(s_1, s_2)$ is the minimal intersection number of $s_1$ and $s_2$. Combining these results, we see that 
$$l(p/q)\geq \sqrt{3}/6\cdot \Delta(p/q, 1/0)=|q|\cdot \sqrt{3}/6.$$ 
Let the parallelogram $P\subset \mathbb R^2$ be the fundamental domain for $T$ corresponding to the generators $1/0$ and $0/1$. In particular $P$ has side-lengths $l(0/1)$ and $l(1/0)$. Think of the edge of $P$ corresponding to $1/0$ as the horizontal base and under this placement let $\text{height}(T)$ be the height of $P$. In particular, $\text{Area}(T)=l(1/0)\cdot \text{height}(T)\leq 6\cdot \text{height}(T)$. Moreover, we see that $l(p/q)\geq |q|\cdot \text{height}(T)$. This implies that 
$$\hat{L}(p/q)\geq |q|\cdot \frac{\sqrt{\text{height}(T)}}{\sqrt{6}}.$$
On the other hand 
$$\hat{L}(p/q)=\frac{l(p/q)}{\sqrt{\text{Area}(T)}}\geq |q|\cdot \frac{\sqrt{3}}{6}\cdot \frac{1}{\sqrt{6\cdot \text{height}(t)}}.$$ Observe that $\max\left\{\frac{h}{6}, \frac{3}{36}\cdot \frac{1}{6\cdot h}\right\}\geq 0.04$ for all $h>0$. Hence $\hat{L}(p/q)\geq 0.2 \cdot |q|$. 
\end{proof}

Let us combine these two results into one more readily accessible corollary. 

\begin{cor} \label{cor:hyp_Dehn_surgery}
    Let $M$ and $N$ be two hyperbolic 3-manifolds with framings on a boundary tori of each such that the corresponding Dehn fillings $M(0/1)$ and $N(0/1)$ are not hyperbolic. We define 
    $$c:=\min\{0.0735, 0.5052\cdot \text{sys}(M)\}\, \textrm{ and } \,D:=\max\left\{10.69, \sqrt{\frac{2\pi}{c}+28.78}\right\}.$$ 
    If $|q|, |q'|>5D$ then, up to isotopy, any diffeomorphism $\varphi\colon M(p/q)\to N(p'/q')$ takes the core curve $\gamma_M$ of the Dehn filling $M(p/q)$ to the core curve $\gamma_N$ of the Dehn filling $N(p'/q')$. 
\end{cor}

\begin{proof} Suppose $|q|, |q'|>5D$. Then, by Lemma \ref{lemma:length_bound}, $\hat{L}(p/q)\geq D$. Furthermore, by Theorem \ref{thm:Dehn_surgery_bounds}, $M(p/q)$ is hyperbolic and the core curve $\gamma_M$ is the only geodesic in $M(p/q)$ of length less than $c$. Similarly, since $|q'|>5D$, $N(p'/q')$ is hyperbolic and the core curve $\gamma_N$ is a geodesic of length less than $c$. By \cite{Gabai_Meyerhoff_Thurston_Isotopy_Hyperbolic}, we know that $\varphi$ is isotopic to an isometry and hence $\varphi^{-1}(\gamma_N)$ must be a geodesic of length less than $c$. Hence, up to isotopy, $\varphi(\gamma_M)=\gamma_N$. 
\end{proof}

\subsection{Existence of strongly-characterising slopes}
We are now ready to prove our main result. 

\begin{thm}
    Any non-trivial knot $K\subset S^3$ has infinitely many strongly-character\-is\-ing slopes. 
\end{thm}
\begin{proof}
    Consider some slope $p/q\in \mathbb Q$ such that $K(p/q)\cong  K'(p/q')$ for some knot $K'\subset S^3$ not isotopic to $K$. Moreover, let $K'$ be such a knot with the fewest JSJ-components. We will begin by making a list of assumptions on $p, q, q'$ until we reach a contradiction. \\

    \noindent
    Let $S^3 \setminus N(K)$ have JSJ-decomposition $Y_0\cup \ldots \cup Y_k$ with outermost JSJ-component $Y_0$. Assume that $p\not \equiv \pm 1 \pmod q$ and that $|q|>2$. Then by Theorem~\ref{thm:JSJ_filling} the JSJ-decomposition of $K(p/q)$ is $Y_0(p/q)\cup Y_1\cup \ldots \cup Y_k$.\\
    
    \noindent
    Let $S^3 \setminus N(K')$ have JSJ-decomposition $Y'_0\cup \ldots\cup Y'_{k'}$ with $Y'_0$ the outermost JSJ-component. Assume that $|q'|>2$, so that by Theorem~\ref{thm:JSJ_filling}, we either have that $K'$ is a cable $K'=C_{r,s}(K'')$ and $K'(p/q')\cong  K''(p/(q's^2)$ or that the JSJ-decomposition of $K'(p/q')$ is $Y'_0(p/q')\cup Y'_1 \cup \ldots \cup Y'_k$. Assume additionally that $|q|$ is prime so that, by Lemma~\ref{lemma:non-cable}, $K(p/q)\not\cong  (C_{r,s}(K))(p/q')$. In particular, in the first case $K''\neq K$, contradicting the assumption that $K'$ is the knot with fewest JSJ-components admitting Dehn filling to $K(p/q)$ and with $K'\neq K$. Hence, $K'(p/q')$ has JSJ-decomposition $Y'_0(p/q')\cup Y'_1\cup \ldots \cup Y'_k$. Finally we asumme that $|q|, |q'|>C_{\text{JSJ}}(K)$, for $C_{\text{JSJ}}(K)$ the constant given by Theorem~\ref{thm:sorya_main}. Hence the diffeomorphism $K(p/q)\cong  K'(p/q')$ maps the surgered pieces $Y_0(p/q)$ and $Y'_0(p/q')$ to each other. \\
    
    \noindent
    Since $Y_0$ is part of a JSJ-decomposition of a knot, it must be one of the cases in Theorem~\ref{thm:Budney_knot_JSJ}. Let us consider these cases separately. \\

    \noindent\textbf{Case 1.} $Y_0$ is the complement of a torus knot $T_{a,b}$ in $S^3$: \\
    By Proposition~\ref{lemma:torus_knot_filling}, $Y_0(p/q)$ is a Seifert fibred space over $S^2$ with singular fibres of orders $a, b$, and $abq-p$. Since $p\not\equiv \pm 1 \pmod p$, $Y_0(p/q)$ has three non-trivial singular fibres and hence by Theorem~\ref{thm:Seifert_unique} a unique Seifert structure up to isomorphism. In particular, $Y'_0(p/q')$ is closed and hence $Y'_0$ is hyperbolic or the complement of a torus knot. If $Y'_0$ is hyperbolic and we additionally assume $|q'|>8$, then $Y'_0(p/q')$ is also hyperbolic. So $Y'_0$ is the complement of a torus knot $T_{u,v}$ and $Y'_0(p/q')$ is a Seifert fibred space over $S^2$ with singular fibres of orders $u, v, uvq'-p$. Since $Y_0(p/q)$ has a unique Seifert structure over an oriented surface up to isomorphism, we know that the multisets $\{a, b, abq-p\}$ and $\{u, v, uvq'-p\}$ are equal. However, our assumption that $K\neq K'$ implies that $\{u,v\}\neq \{a,b\}$. Hence, possibly swapping $u,v$, we have 
    $$u=a, v=abq-p, uvq'-p=b  \text{ or } u=b, v=abq-p, uvq'-p=a, $$
    which implies that 
    $$ q'=-\frac{p+b}{ap-a^2bq} \text{ or } q'=-\frac{p+a}{bp-ab^2q}.$$
    Since, $a, b>1$ there exists some constant $C(a,b,q)$, depending only on $a,b$, and $q$, such that if $p>C(a,b,q)$ then in either case $-1<q'<1$. This is a contradiction, since $q'$ is a non-zero integer. So let us assume $p>C(a,b,q)$. \\

    \noindent\textbf{Case 2.} $Y_0$ is a composing space:\\
    By Theorem~\ref{thm:knot_JSJ_filling_component}, $Y_0(p/q)\cong  Y'_0(p/q')$ is a Seifert fibred space over an oriented surface. Since $Y_0(p/q)$ has at least two boundary components, $Y'_0$ is not a torus knot or cable space. If $Y'_0$ were a hyperbolic space, since we assumed $|q'|>2$, $Y'_0(p/q')$ would also be hyperbolic. So $Y'_0$ must also be a composing space. Hence, by Theorem~\ref{thm:knot_JSJ_filling_component}, we know that $Y_0(p/q), Y'_0(p/q')$ are both Seifert fibred over an oriented surface with exactly one singular fibre of orders $|q|, |q'|$ respectively. Moreover, these singular fibres are the core curves of the Dehn filling. Finally, by Theorem~\ref{thm:Seifert_unique}, these Seifert fibrations are unique up to isotopy and hence $|q|=|q'|$ and the two core curves are isotopic. Hence $K=K'$, a contradiction.\\
    
    \noindent\textbf{Case 3.} $Y_0$ is hyperbolic: \\
    By Theorem ~\ref{thm:knot_JSJ_filling_component}, $Y_0(p/q)\cong  Y'_0(p/q')$ is a hyperbolic space. Hence, $Y'_0$ must also be a hyperbolic space. Let $\varphi\colon K(p/q)\to K'(p/q')$ be a diffeomorphism between the two Dehn fillings. We have already assumed, that $\varphi$ maps the surgered pieces $Y_0(p/q)$ and $Y'_0(p/q')$ to each other, so consider the restriction $\varphi|_{Y_0(p/q)}$. By Corollary \ref{cor:hyp_Dehn_surgery}, there exists some $C_{\text{hyp}}(Y_0)$, depending only on $Y_0$ such that if $|q|, |q'|>C_{\text{hyp}}(Y_0)$, then, up to isotopy, $\varphi|_{Y_0(p/q)}$ takes the core curve of the Dehn filling in $Y_0(p/q)$ to the core curve of the Dehn filling of $Y'_0(p/q')$. In particular, $K=K'$, a contradiction.\\

    \noindent\textbf{Case 4.} $Y_0$ is a cable space $C_{r,s}$:\\
    Since $Y_0(p/q)\cong  Y'_0(p/q')$ has exactly one boundary component, we know that $Y'_0$ must be hyperbolic or a cable space. If $Y'_0$ were hyperbolic, then $Y'_0(p/q')$ would also be hyperbolic, a contradiction. So $Y'_0$ is a cable space $C_{r', s'}$. Now observe that $Y'_1\cup \ldots\cup Y'_k\cong  Y_1\cup \ldots\cup Y_k\subset S^3$ has one boundary component, so is diffeomorphic to $S^3\setminus N(J)$ for some knot $J$. In particular $K=C_{r,s}(J), K'=C_{r', s'}(J)$ are both cables. The Seifert fibration on $Y_0(p/q)$ is unique up to isotopy and is an extension of the Seifert fibration on $Y_0$. By Example~\ref{ex:SFS_cable}, the Seifert fibration on $Y_0$ has  fibre in $\partial Y_0(p/q)$ of slope $r/s$ with respect to the Seifert framing given by $S^3\setminus N(J)$. However, by uniqueness, this is also the slope $r'/s'$ of a fibre in the Seifert fibration of $Y'_0$ on $\partial Y'_0(p/q')$. Hence $r/s=r'/s'$ and thus $K=K'$, a contradiction. \\
    
    \noindent 
    In summary we know that if $K(p/q)\cong  K'(p/q')$ for $K'\neq K$ and $K'$ having as few JSJ-components as possible, then $p,q, q'$ cannot satisfy all of the following: 
    \begin{enumerate}
\item $|q|, |q'|>2$, 
\item $|q'|>8$ if $K$ is a torus knot, 
 \item $p\not \equiv \pm 1 \pmod q$, 
\item $q$ is prime, 
\item $|q|, |q'|>C_{\text{JSJ}}(K)$,
\item if $K$ is a torus knot $T_{a,b}$ then $p>C(a,b,q)$ as given in Case 1 above,
\item if $Y_0$ is hyperbolic, then $|q|, |q'|>C_{\text{hyp}}(Y_0)$, as given in Case 3 above.\\
\end{enumerate}

\noindent
To conclude the proof, we will show there are infinitely many pairs $p,q$ such that if $K(p/q)\cong  K'(p/q')$ then $p, q, q'$ satisfy all 6 conditions above. Let $C(K)$ be the maximum of $C_{\text{JSJ}}(K)$, $C_{\text{hyp}}(Y_0)$ if $Y_0$ is hyperbolic, and $8$. Let $p_1, \ldots, p_k$ denote the set of all odd primes that are less or equal to $C(K)$. Let $q>C(K)$ be some additional prime with $q\equiv 1 \pmod 4$ and let $r$ be some quadratic non-residue mod $q$. Then by the Chinese remainder theorem and Dirichlet's theorem on primes in arithmetic progressions, we may find infinitely many primes $p$ satisfying the following: 
\begin{enumerate}
    \item $p\equiv 1\pmod 8$, 
    \item $p\equiv 1\pmod {p_i}, i=1, \ldots, k$, 
    \item $p\equiv r\pmod q$, 
   \item if $K$ is a torus knot $T_{a,b}$ then $p>C(a,b,q)$. 
\end{enumerate}
Fix such $p,q$. Then $p\equiv 1 \pmod 4$ and $p$ is a quadratic residue mod $p_i, i=1,\ldots, k$. Hence, by quadratic reciprocity, $p_i, i=1\ldots,k$, are all quadratic residues mod $p$. Since $p\equiv 1 \pmod 8$ both $-1$ and $2$ are quadratic residues mod $p$. Thus in summary, any integer $z$ with $|z|\leq C(K)$ is a product of quadratic residues mod $p$ and hence itself a residue mod $p$. \\
By Proposition~\ref{prop:linking_framing}, we know that if $K(p/q)\cong  K(p/q')$ then $q\cdot q'$ is a quadratic residue mod $p$. Moreover, by quadratic reciprocity, $q$ is not a residue mod $p$. Hence $q'$ must be a non-residue and hence $|q'|>C(K)$. Finally, since $p$ is a non-residue mod $q$, we know $p\not\equiv \pm 1 \pmod q$. In summary, this choice of $p,q$ satisfies all 6 conditions in the discussion above, a contradiction. Thus, we have infinitely many strongly-characterising slopes $p/q$ for $K$. 
\end{proof}

\begin{rem}
The constants $C(a,b,q)$, $C_{\text{JSJ}}(K)$, and $C_{\text{hyp}}(K)$ are computable and details can be found in \cite{Wakelin_Sorya_effective_char}. Moreover, there are effective versions of Dirichlet's theorem on primes in arithmetic progressions~\cite{Linnik_Dirichlet_1, Linnik_Dirichlet_2}. Hence, although the authors avoid doing this, these results can be made effective. 
\end{rem}

\section{Integer slopes}\label{section:integer_slopes}
As discussed in the introduction, integer slopes play a special role when discussing Dehn surgery, and we wish to find examples of integer strongly-characterising slopes. In this section, we show that sufficiently large integer slopes being characterising for a knot $K$ is closely related to sufficiently large integer slopes being strongly-characterising for $K$. To be precise, we show the following two results.

\begin{prop}\label{prop_integer_1}
    Consider a hyperbolic knot $K\subset S^3$ such that every sufficiently large integer slope is characterising for $K$. Then there are infinitely many integer strongly-characterising slopes for $K$.
\end{prop}

McCoy showed that for a hyperbolic L-space knot $K$, all sufficiently large integer slopes are characterising \cite[Theorem 3]{McCoy_char_hyp}. Hence Theorem \ref{thm:integer_L_space} follows from Proposition \ref{prop_integer_1}. 

\begin{prop}\label{prop_integer_2}
    Consider a hyperbolic knot $K\subset S^3$ such that every sufficiently large integer slope is characterising for $K$. Then for every sufficiently large integer slope $p$ we have $K(p)\neq K'(p/q')$ for $K'\neq K$ and $q'>0$.
\end{prop}

For the proof of Proposition~\ref{prop_integer_1}, we will use the Casson--Walker invariant $\lambda(M^3)\in \mathbb Q$ for rational homology $3$-spheres $M$. For an introduction to the Casson--Walker invariant, we refer to \cite{Into_Casson_Walker}. The main requisite property of $\lambda$ will be the following result of Boyer and Lines.

\begin{thm} [Boyer--Lines \cite{Surgery_Formula_Casson_Walker}]\label{thm:lambda_surgery_formula} For every knot $K\subset S^3$ and slope $p/q\neq 0$, we have 
$$\lambda(K(p/q))=\lambda(L(p,q))+\frac{q}{2p}\Delta''_K(1),$$
 where $\Delta_K(t)$ is the Alexander polynomial of $K$ normalised to be symmetric and satisfy $\Delta_K(1)=1$.\qed
\end{thm}

Let us remark that $\lambda(L(p,q))=-\lambda(L(p,-q))$. Moreover, we have the following result of Rasmussen.

 \begin{prop}[Rasmussen, Lemma 4.3 \cite{Lens_space_Casson_Walker}]\label{prop:lambda_lens} 
 \label{lambda_lens_formula}
     Let $a_i\geq 2$ such that
     $$\frac{p}{q}=a_1-\frac 1{a_2-\frac1{\ddots -\frac{1}{a_n}}}.$$
     Then we have
     $$\lambda(L(p,q))=-\frac{1}{24}\left(\frac qp+\frac rp+\sum_{i=1}^n(a_i-3)\right),$$
     where $0<r<p$ denotes the inverse of $q \pmod p$. \qed
 \end{prop}

 \begin{rem}
    In this article the convention is that $L(p,q)$ denotes the $p/q$ surgery along the unknot and that $\lambda(L(3,1))=-1/36$. However, the reader should be aware that the convention chosen in \cite{Lens_space_Casson_Walker} is that $L(p,q)$ denotes the $-p/q$ surgery along the unknot and $\lambda(L(3,1))=-1/18$. 
 \end{rem}

\begin{proof} [Proof of Proposition~\ref{prop_integer_1}]
Take $p>1$ large enough such that $K(p)$ is hyperbolic and suppose $K(p)\cong K'(p/q')$ for $K'\neq K$. If $|q'|>2$, then by Lemma~\ref{lemma:hyp_reduction}, we have $K(p/q)\cong  K''(p/(q'\cdot w^2))$ for some hyperbolic knot $K''\subset S^3$ and some $w\in \mathbb N$. Moreover, if $K'\neq K''$ we may take $w>1$. Recall that by Theorem~\ref{thm:lackenby_char_slope}, there is a constant $C(K)$, depending only on $K$, such that if $|q'w^2|\geq C(K)$ then $K=K''$ and $p=p/(q'w^2)$. Hence, we know that $|q'|<C(K)$. Take $p$ to be divisible by all integers less than $C(K)$. Since $\gcd(p, q')=1$, we see that $q'=\pm 1$. Since a sufficiently large $p$ is characterising for $K$, we may assume $q'=-1$. \\

\noindent
If $K(p)\cong  K'(-p)$, Theorem~\ref{thm:lambda_surgery_formula} shows that
\begin{equation*}
    \lambda(L(p,1))+\frac{1}{2p}\Delta''_K(1)=\lambda(L(p, -1))-\frac{1}{2p}\Delta''_{K'}(1).
\end{equation*}
Plugging in the values from Proposition~\ref{prop:lambda_lens} yields
\begin{equation*}
\frac{\Delta''_K(1)+\Delta''_{K'}(1)}{2p}=\frac{2}{24} \left(\frac{1}{p}+\frac{1}{p}+(p-3)\right),
\end{equation*}
which simplifies to
\begin{equation*}
6\big(\Delta''_K(1)+\Delta''_{K'}(1)\big) =p^2-3p+2.
\end{equation*}
This is impossible for $p$ divisible by $3$. Hence, for $p$ sufficiently large and divisible by sufficiently many integers, the slope $p$ is strongly-characterising for $K$.
\end{proof}

For the proof of Proposition~\ref{prop_integer_2} we will need the Heegaard Floer $d$-invariant, which associates to a $\mathbb Q$-homology sphere $M^3$ and $\mathfrak{s}\in \text{Spin}^c(M^3)$ the rational number $d(M^3, \mathfrak{s})$. For an introduction to the d-invariant, we refer to \cite{OS_HF_Lectures}. In particular, a Dehn filling $K(p/q)$ induces a preferred identification of $\{0, \ldots, p-1\}$ with $\text{Spin}^c(K(p/q))$. The following result allows us to compute $d(K(p/q), \mathfrak{s})$.

\begin{thm}[Ni--Wu, Proposition 1.6 \cite{Ni_Cosmetic}]\label{thm:d_surgery} Consider a knot $K\subset S^3$, a slope $p/q>0$, and $\mathfrak{s}\in \text{Spin}^c(M)$ corresponding to $i\in \{0, \ldots, p-1\}$, then 
$$d(K(p/q), i)=d(p,q,i)-2\max\left\{V_{\lfloor \frac i q\rfloor}(K), V_{\lceil \frac{p-i}q\rceil}(K)\right\},$$
for $V_j(K)\geq 0$ a non-increasing sequence of integers only depending on $K$. Moreover, we can recursively compute
$$d(p,q,i)=-\frac 14 +\frac{(p+q-1-2i)^2}{4pq}-d(q,r,i'),$$
where $q\equiv r\pmod p, i\equiv i'\pmod q,$ and $d(1,0,0)=d(S^3)=0$. \qed
\end{thm}

\begin{proof} [Proof of Proposition~\ref{prop_integer_2}]
Suppose that for arbitrarily large $p$ there exists a knot $K'\neq K$ and $q'>0$ such that $K(p)\cong  K'(p/q')$. Take such a $p>1$ large enough so that $K(p)$ is hyperbolic and suppose $K(p)\cong  K'(p/q')$ for $K'\neq K$. If $|q'|>2$, then Lemma~\ref{lemma:hyp_reduction} implies $K(p/q)\cong  K''(p/(q'\cdot w^2))$ for some hyperbolic knot $K''\subset S^3$ and some $w\in \mathbb N$. Moreover, if $K'\neq K''$ we may take $w>1$. Recall that by Theorem~\ref{thm:lackenby_char_slope}, there is a constant $C(K)$, depending only on $K$, such that if $|q'w^2|\geq C(K)$ then $K=K''$ and $p=p/(q'w^2)$. Hence, we know that $0<q'<C(K)$. Moreover, since $p$ is characterising for $p$ sufficiently large, we may assume that $1<q'<C(K)$.\\

\noindent
Recall that the Dehn filling $K(p)$ gives an identification $\{0, \ldots, p-1\}\leftrightarrow\text{Spin}^c(K(p))$. Let $\mathfrak{s}\in \text{Spin}^c(K(p))$ correspond to $0$ under this identification. Then Theorem~\ref{thm:d_surgery} yields $$d(K(p), \mathfrak{s})=d(p,1,0)-2\max\left\{V_{0}(K), V_{p}(K)\right\}=d(p,1,0)-2V_0(K).$$
Let $j\in \{0, \ldots, p-1\}$ correspond to $\mathfrak{s}$ under the identification $\{0, \ldots, p-1\}\leftrightarrow \text{Spin}^c(K'(p/q'))\cong  \text{Spin}^c(K(p))$ given by the Dehn filling $K'(p/q')$. Thus by Theorem~\ref{thm:d_surgery} we have that
$$d(K'(p/q'), \mathfrak{s})=d(p,q',j)-2\max\left\{V_{\lfloor \frac j {q'}\rfloor}(K'), V_{\lceil \frac{p-j}{q'}\rceil}(K')\right\}.$$
Comparing these expressions for the $d$-invariants of $K(p)\cong  K'(p/q')$ we obtain
$$
2\max\left\{V_{\lfloor \frac j {q'}\rfloor}(K'), V_{\lceil \frac{p-j}{q'}\rceil}(K')\right\}-2V_0(K)=d(p, q', j)-d(p, 1, 0).
$$ 
Recall that the $V_i(K')$ are always non-negative. Hence for arbitrarily large $p$, we have found $j\in \{0, \ldots, p-1\}$ and $1<q'<C(K)$ such that $d(p, q', j)-d(p, 1, 0)>-2V_0(K)$. On the other hand, by Theorem~\ref{thm:d_surgery} we deduce that for $r'\equiv q'\pmod p$ and $j'\equiv j\pmod {q'}$ we have
\begin{eqnarray*}
d(p, q', j)-d(p, 1, 0)&=&\frac{(p+q'-1-2j)^2}{4pq'}-\frac{p^2}{4p}-d(q', r', j')+d(1, 1, 0)\\
&\leq & \frac{(p+q'-1)^2}{4pq'}-\frac{p}{4}-d(q', r', j')+d(1, 1, 0)\\
&\leq & \frac{p}{4q'}-\frac{p}{4}+\frac{(q'-1)}{2q'}+\frac{(q'-1)^2}{4pq'}-d(q', r', j')+d(1, 1, 0).
\end{eqnarray*}
By construction $r', j'\leq q'$, and $1<q'< C(K)$. Hence $$\left|\frac{q'-1}{2q'}+\frac{(q'-1)^2}{4pq'}-d(q',r', j')+d(1,1,0)\right|$$ is bounded by a constant depending only on $K$. Moreover, since $q'>1$ we see that $\frac{p}{4q'}-\frac{p}{4}\to -\infty$ as $p\to \infty$. Hence, for sufficiently large $p$ and any $j\in \{0, \ldots, p-1\}$ we see that $d(p, q', k)-d(p, 1, 0)\leq -2V_0(K)$, a contradiction. 
\end{proof}

 \let\MRhref\undefined
 \bibliographystyle{hamsalpha}
 \bibliography{lit.bib}

\end{document}